\documentclass[sn-mathphys-num]{sn-jnl}% Math and Physical Sciences Numbered Reference Style 
\usepackage{graphicx}%
\usepackage{multirow}%
\usepackage{amsmath,amssymb,amsfonts}%
\usepackage{amsthm}%
\usepackage{lmodern}
\usepackage{mathrsfs}%
\usepackage[title]{appendix}%
\usepackage{xcolor}%
\usepackage{textcomp}%
\usepackage{manyfoot}%
\usepackage{booktabs}%
\usepackage{algorithm}%
\usepackage{algorithmicx}%
\usepackage{algpseudocode}%
\usepackage{listings}%
\usepackage{amsfonts}
\usepackage{amsmath,amsthm,amssymb}
\usepackage[all,dvips]{xy}
\usepackage{eucal}
\usepackage{comment}
\usepackage{tikz-cd}

\newcommand{\C}{\mathbb{C}}
\DeclareMathAlphabet{\mathnormal}{OT1}{cmr}{m}{it}
\newtheorem{theorem}{Theorem}[section]
\newtheorem{lemma}[theorem]{Lemma}
\newtheorem{corollary}[theorem]{Corollary}
\newtheorem{conjecture}[theorem]{Conjecture}
\newtheorem{proposition}[theorem]{Proposition}
\newtheorem*{theorem*}{Theorem}
\newtheorem*{corollary*}{Corollary}
\theoremstyle{definition}
\newtheorem{defn}[theorem]{Definition}

\newtheorem{remark}[theorem]{Remark}

\newtheorem{example}[theorem]{Example}

\numberwithin{equation}{section}
\begin{document}

\title[Article Title]{An Inductive Approach to Basepoint-Freeness of Adjoint Series on Irregular Varieties}

%%=============================================================%%
%% GivenName	-> \fnm{Joergen W.}
%% Particle	-> \spfx{van der} -> surname prefix
%% FamilyName	-> \sur{Ploeg}
%% Suffix	-> \sfx{IV}
%% \author*[1,2]{\fnm{Joergen W.} \spfx{van der} \sur{Ploeg} 
%%  \sfx{IV}}\email{iauthor@gmail.com}
%%=============================================================%%

\author*[1]{\fnm{Houari} \sur{Benammar Ammar}}\email{benammar\_ammar.houari@courrier.uqam.ca}

\affil*[1]{\orgdiv{Département de Mathématiques}, \orgname{UQAM}, \orgaddress{\street{201, Avenue du Président-Kennedy}, \city{Montréal}, \postcode{H2X 3Y7}, \state{Québec}, \country{Canada}}}

%%==================================%%
%% Sample for unstructured abstract %%
%%==================================%%
\abstract{In this paper, we show how to prove the basepoint-freeness for linear systems on irregular varieties inductively. For instance, we prove that Fujita's conjecture holds for irregular varieties of dimension  $\mathnormal{n}$ with a nef anticanonical bundle assuming it holds for lower-dimensional varieties and under mild conditions. \\

 
 ORCID: 0000-0002-4645-8362}

%%\pacs[JEL Classification]{D8, H51}

%%\pacs[MSC Classification]{35A01, 65L10, 65L12, 65L20, 65L70}

\maketitle
\section{Introduction}
\par Let $D$ be a positive divisor on a smooth complex projective irregular variety $X$ of dimension $n$. In this paper, we shall use the Albanese map to prove that the basepoint-freeness problem  for the adjoint linear system $|K_{X} + cD|$, where $c$ is a positive integer, can be reduced to a lower-dimensional cases. 
\par To place matters into context, let us recall that one celebrated problem for the basepoint-freeness of adjoint linear systems is  Fujita's conjecture \cite{fujita}.
\begin{conjecture}[{\cite[Page 167]{fujita}}]\label{fujitaconjecture}
    Let $X$ be a complex smooth projective variety of dimension $n$, and let $D$ be an ample divisor on $X$. Then, the linear system $|K_{X} + mD|$ is basepoint-free for all $m \geq n+1$. 
\end{conjecture}
This conjecture is trivial in dimension $1$ and was proved by Reider \cite[Theorem 1]{reider} in dimension $2$. For $\dim X = 3$ or $4$, it was proved by Ein-Lazarsfeld \cite[Corollary 2*]{lazarsfeld} and Kawamata \cite[Theorem 4.1]{kawamatafujita}, respectively. Recently, the conjecture was settled in dimension $5$ by Ye and Zhu, see \cite[Page 3]{feizhu}. In higher dimensions, there exist some partial results. For instance, Demailly in \cite[Page 324]{demailly1} and \cite[Theorem 0.2]{demailly2} established certain results using analytic techniques and Monge–Ampère equations, he showed that $2K_{X} + mD$ is very ample if $m \geq 2 + \binom{3n+1}{n}$ for every ample divisor $D$ on $X$. In \cite[Corollary 0.2]{siu}, Angehrn and Siu achieved an important result by utilizing multiplier ideal sheaves, Nadel’s vanishing theorem, and the Ohsawa–Takegoshi extension theorem. Specifically, they proved that Conjecture \ref{fujitaconjecture} holds if $m \geq \frac{n(n+1)}{2}$. Later, Heier \cite[Theorem 1.4]{Heier} improved this bound to $O(n^{\frac{4}{3}})$. Some logarithmic bounds were obtained by Ghidelli and Lacini \cite[Theorem 1.1]{lacini}. Notably, Helmke \cite{helmke1} and \cite[Page 3]{helmke2} proved Conjecture \ref{fujitaconjecture} under stronger numerical conditions on $D$. 
\par Recall that by the classical Castelnuovo-Mumford regularity theorem, if $D$ is ample and globally generated, then Conjecture \ref{fujitaconjecture} holds (\cite[Example 1.8.23]{Lazabook}). \par In this article, we mainly study Fujita’s Conjecture for irregular varieties. We first mention that, in \cite[Section 5]{popapareschi1}, Pareschi and Popa, for instance, obtained some basepoint-freeness results on varieties whose Albanese morphism is finite. In Section 3, we present new inductive results on Fujita’s Conjecture for irregular varieties. The main results are as follows:
\begin{theorem*}[Theorem \ref{maintheorem1}]
Let $X$ be an irregular variety of dimension $n\geq 2$ with Albanese dimension $1 \leq \alpha(X)<n$. Let $X \xrightarrow{f} Z \xrightarrow{u} \operatorname{alb}(X) \subseteq \operatorname{Alb}(X)$ be the Stein factorization of the Albanese morphism $\operatorname{alb}$, and let $F$ be a general fiber of the morphism $f$. Let $D$ be an ample divisor on $X$.
\par If  Conjecture \ref{fujitaconjecture} holds in dimension $<n$, then $|K_{X} + mD + \operatorname{alb}^{*}p|$ has no  basepoint supported on $F$ for all $m \geq n-\alpha(X)+1$ and for general $p \in \operatorname{Pic}^{0}(\operatorname{Alb}(X))$.
\par Additionally, if the following condition is satisfied:
\begin{itemize}
    \item[$(*)$]There exists an integer $r$ with $1\leq r \leq \alpha(X)$ such that $|rD_{|_{F}}|$ is basepoint-free and $rD - K_{X}$ is nef and big,
\end{itemize}
then $|K_{X} + mD|$ has no  basepoint supported on $F$ for all $m \geq n+1$.
\end{theorem*}
\par We can include the case of $X$ of maximal Albanese dimension in Theorem \ref{maintheorem1}. However, we decided to separate this case from the inductive case
(in other words, from the positive-dimensional fiber). Theorem \ref{mainresult3} covers the case of maximal Albanese dimension.
\par In some special situations, such as in the case of varieties with $-K_{X}$ nef, we can prove Conjecture $\ref{fujitaconjecture}$ by induction under mild conditions. Recall that projective varieties with $-K_{X}$ nef are a larger class than Fano varieties. One of the methods to study these varieties is to analyze their Albanese maps if they are irregular. In \cite[Theorem 1.2]{Cao}, Cao proved that for an irregular variety $X$ with $-K_{X}$ nef, the Albanese map $\operatorname{alb}: X \to \operatorname{Alb}(X)$ is a locally trivial fibration. This result was conjectured in \cite[Page 221]{DPS}, many authors contributed 
to solving the problem and proved some partial results, for instance, in \cite[Page 532]{steven} and \cite[Theorem 1]{qizhang1}. 

Further, Cao and Höring \cite[Theorem 1.4]{caohoring} proved a decomposition structure theorem for varieties with $-K_{X}$ nef. They showed that the universal cover $\widetilde{X}$ of $X$ decomposes as the following product:
\[
\widetilde{X} \simeq \C^{d} \times \prod Y_{j} \times \prod S_{k} \times Z,
\]
where $Y_{j}$ are irreducible projective Calabi-Yau varieties, $S_{k}$ are irreducible projective hyperkähler varieties, and $Z$ is a projective rationally connected variety with $-K_{Z}$ nef. Now, we state the following theorem for these type of varieties. 
\begin{theorem*}[Theorem \ref{supplement}]
Let $X$ be an irregular variety of dimension $n \geq 2$ with $-K_{X}$ nef. Let $\operatorname{alb}: X \to  \operatorname{Alb}(X)$ be the Albanese map, and let $D$ be an ample divisor on $X$. If Conjecture \ref{fujitaconjecture} holds in dimension $<n$ and there exists an integer $r$ with $1\leq r \leq \alpha(X)$ such that $|rD_{|_{F}}|$ is basepoint-free for every fiber $F$ of $\operatorname{alb}$, then Conjecture \ref{fujitaconjecture} holds for $X$.
\end{theorem*}
\textbf{}

\textbf{Organization of the paper.}
\par In Section 2, after recalling preliminaries, we provide  comparisons between the notion of a sheaf being \emph{continuously globally generated} (Definition \ref{continousdefinition}) and the notion of a sheaf  having no \emph{essential basepoints} (Definition \ref{essantialdefinition}). While these results are not necessarily needed in the proof of Theorem \ref{maintheorem1} and Theorem \ref{supplement}, they are of some independent interest and are given in Proposition \ref{premierproposition} and Proposition \ref{deuxiemeproposition}. 
\par In Section 3, we study Fujita's Conjecture \ref{fujitaconjecture} on irregular varieties. We obtain some interesting results  using the Albanese map, and by assuming the conjecture holds in lower dimensions. In particular, we provide the proofs of Theorem \ref{maintheorem1} and Theorem \ref{supplement}.
    \par In Section 4, we  study the basepoint-freeness of adjoint series for varieties
fibered over abelian varieties, where we state Proposition \ref{kollarvanishing} and Theorem \ref{supplement3}.
\par In Section 5, we explore the basepoint-freeness of adjoint series for varieties
of maximal Albanese dimension. In this case, the Albanese map $\operatorname{alb}$ is generically finite. As a consequence, we obtain Theorem \ref{mainresult3}.
\section{Preliminaries}\label{sec2}
\textbf{}
\par\textbf{Notations.} In what follows, unless otherwise specified, $X$ is a smooth complex projective irregular variety of dimension $n\geq 2$, $Y$  is an abelian variety of dimension $g$, and $\mathcal{F}$ is a nonzero coherent sheaf. We collect below several notations that will be also used frequently throughout the paper. In particular, 
\begin{itemize}
    \item[]\( K_X \): the canonical divisor on \( X \). 
  \item[] \( p \in \operatorname{Pic}^0(X) \): a point in the Picard variety of \( X \), where
    \[
    \operatorname{Pic}^0(X) := \{ p \in \operatorname{Pic}(X) \mid p \text{ is topologically trivial} \},
    \]
    that is, the connected component of the identity consisting of line bundles that are topologically trivial.
    \item[] \( p^\vee \): the dual of \( p \).
    \item[] Without loss of clarity, we use the same symbol $p$ to denote both the line bundle and its corresponding divisor on $X$.
    \item[] $F$ is a general fiber of the morphisms used below.
    \item[] \( D_{|_F} \): the restriction of the divisor \( D \) to  \( F \).
    \item[] \( \C(x) \): the skyscraper sheaf at $x$.
    \item[] $\operatorname{Coh}(.)$: For any algebraic variety $S$, $\operatorname{Coh}(S)$ denotes the abelian category of coherent sheaves on $S$.
    \item[] $\operatorname{alb}$: the Albanese map.
    \item[] $\operatorname{alb}(X)$: the image of the Albanese map.
    \item[] $\operatorname{Alb}(X)$: the Albanese variety of $X$.
    \item[] $\alpha(X)$: the Albanese dimension of $X$ (Definition \ref{dimensionalbanese}).
    \end{itemize}
\par We briefly  recall some basic definitions, we refer to \cite{popapareschi1}, \cite{haconchen}, \cite{Greenlazarsfeld},  \cite{Nathan}, 
\cite{popasurvey} and \cite{popapareschi} for more details. 
\begin{defn}[{\cite[Definition 2.3]{popasurvey}}]\label{itdefn}
    A coherent sheaf $\mathcal{F}$ on an abelian variety $Y$ with $\dim Y = g$ satisfies \emph{IT with index $0$} if 
\[
H^{i}(Y, \mathcal{F} \otimes p) = 0 \hspace{0.2cm} \text{for all } p \in \operatorname{Pic}^{0}(Y) \hspace{0.2cm} \text{and for all } i > 0.
\]
\end{defn}
\begin{defn}[{\cite[Definition 2.10]{popapareschi1}}]\label{continousdefinition}
    Let $X$ be an irregular variety, and let $\mathcal{F}$ be a coherent sheaf on $X$. We say that $\mathcal{F}$ is \emph{continuously globally generated} if for any nonempty open subset $U \subseteq \operatorname{Pic}^{0}(X)$,  the sum of
evaluation maps
\[
\bigoplus_{p \in U} H^{0}(X, \mathcal{F} \otimes p) \otimes p^{\vee} \to \mathcal{F}
\]
is surjective.
\end{defn}
\par We recall the Fourier-Mukai setting. We refer to Mukai \cite{Mukai} for more details. We denote by $\mathcal{P}$ the Poincaré line bundle on $Y \times \operatorname{Pic}^{0}(Y)$, and $Y$ is an abelian variety as before. For any coherent sheaf $\mathcal{F}$ on $Y$, we can associate the sheaf ${p_{2}}_{*}(p_{1}^{*}\mathcal{F} \otimes \mathcal{P})$ on $\operatorname{Pic}^{0}(Y)$ where
$p_{1}$ and $p_{2}$ are the natural projections on $Y$ and $\operatorname{Pic}^{0}(Y)$, respectively. This correspondence gives a
functor
\[
\Hat{S}: \operatorname{Coh}(Y) \to \operatorname{Coh}(\operatorname{Pic}^{0}(Y)).
\]
If we denote by $\operatorname{D}(Y)$ and $\operatorname{D}(\operatorname{Pic}^{0}(Y))$ the bounded derived categories of $\operatorname{Coh}(Y)$ and
$\operatorname{Coh}(\operatorname{Pic}^{0}(Y))$, then the derived functor $$\mathcal{R}\Hat{S}: \operatorname{D}(Y) \to \operatorname{D}(\operatorname{Pic}^{0}(Y))$$  is defined and called the
Fourier-Mukai functor. Similarly, we consider $$\mathcal{R}S : \operatorname{D}(\operatorname{Pic}^{0}(Y))  \to \operatorname{D}(Y)$$ in a similar
way. According to the celebrated result of Mukai \cite[Theorem 2.2]{Mukai}, the Fourier-Mukai functor induces an equivalence of categories between the two derived categories  $\operatorname{D}(Y)$ and $\operatorname{D}(\operatorname{Pic}^{0}(Y))$. More precisely, we have
\[
\mathcal{R}S \circ \mathcal{R}\Hat{S} \simeq (-1_{Y})^{*}[-g]
\]
and 
\[
\mathcal{R}\Hat{S} \circ \mathcal{R}S \simeq (-1_{\operatorname{Pic}^{0}(Y)})^{*}[-g]
\]
where $[-g]$ is  an operation that shifts the complex $g$ places
to the right.
\par Now,  we describe the cohomological support locus \cite{Greenlazarsfeld}: 
\begin{defn}[{\cite[Page 389]{Greenlazarsfeld}}]
Let $\mathcal{F}$ be a coherent sheaf on an abelian variety $Y$ of dimension $g$. The \emph{$i$-th cohomological support locus} $V^{i}(\mathcal{F})$ of $\mathcal{F}$ is defined to be:
    \[
    V^{i}(\mathcal{F}) := \{p \in \operatorname{Pic}^{0}(Y)| H^{i}(Y, \mathcal{F} \otimes p) \neq 0\}, \hspace{0.2cm} \text{for all }   0\leq i \leq g .\]
\end{defn}
  \par The cohomological support locus are studied very carefully in   \cite{Greenlazarsfeld} and \cite{simpson}. There  the authors proved a foundational generic vanishing theorem. 
\begin{remark}
    By base change, the following inclusion always holds: $$\operatorname{Supp}(\mathcal{R}^{i}\Hat{S}(\mathcal{F})) \subseteq V^{i}(\mathcal{F}).$$
    \end{remark}
    \begin{defn}[{\cite[Definition 3.1]{popasurvey}}] A coherent sheaf $\mathcal{F}$ on $Y$ is called \emph{M-regular} if 
 \[
\text{codim}_{\operatorname{Pic}^{0}(Y)}\text{(Supp}(\mathcal{R}^{i}\Hat{S}(\mathcal{F}))) > i \hspace{0.2cm} \text{for all }  i \geq 1.
\]
A coherent sheaf $\mathcal{F}$ on $Y$ is called a \emph{generic vanishing sheaf} or a \emph{GV-sheaf} if its cohomological support locus $V^{i}(\mathcal{F})$ satisfies the following inequality: 
 \[
    \text{codim}_{\operatorname{Pic}^{0}(Y)}(V^{i}(\mathcal{F})) \geq i \hspace{0.2cm} \text{for all }  i \geq 1.
    \]
    \par We remark that $M$-regularity is achieved, in particular, if
    \[
    \text{codim}_{\operatorname{Pic}^{0}(Y)}(V^{i}(\mathcal{F})) > i \hspace{0.2cm} \text{for all } i \geq 1.
    \]
\end{defn}
\begin{remark}\label{remarkproperties}
If $\mathcal{F}$ is an $M$-regular sheaf on an abelian variety $Y$, then $\mathcal{F}$ is a $GV$-sheaf. Furthermore, if $\mathcal{F}$ satisfies $IT$  with index $0$, then it is clearly  $M$-regular. In the following proposition,  we recall that the $M$-regular property is stronger than the notion of being continuously globally generated.
\end{remark}
\begin{proposition}[{\cite[Proposition 2.13]{popapareschi1}}]\label{propoMregularity}
    Every $M$-regular coherent sheaf $\mathcal{F}$ on an abelian variety $Y$ is continuously globally generated.
\end{proposition}
\begin{defn}[{\cite[Definition 2.2]{haconchen}}]\label{essantialdefinition}
A coherent sheaf $\mathcal{F}$ on an irregular variety $X$ is said to have an \emph{essential basepoint} at $x$ if there exists a surjective map $\mathcal{F} \to \C(x)$ such that for all $p \in \operatorname{Pic}^{0}(X)$, the induced map $H^{0}(X, \mathcal{F} \otimes p ) \to H^{0}(X, \C(x))$ is zero.
\end{defn}
\par In \cite{haconchen}, the authors proved the next proposition.
\begin{proposition}[{\cite[Proposition 2.3]{haconchen}}]\label{haconchenproposition}
If $\mathcal{F}$ is a coherent  sheaf on an abelian variety $Y$  that satisfies $IT$ with index $0$, then $\mathcal{F}$ has no essential basepoints.
\end{proposition}
\par The proof of the main theorems does not rely on the following propositions, but we are interested in comparing Definition \ref{continousdefinition} and Definition \ref{essantialdefinition}.
\begin{proposition}\label{premierproposition}
Let $\mathcal{F}$ be a coherent sheaf on an irregular variety $X$. If $\mathcal{F}$ is continuously globally generated, then $\mathcal{F}$ has no essential basepoints.
\end{proposition}
\begin{proof}
If $\mathcal{F}$ is continuously globally generated, then for any nonempty open subset $U \subseteq \operatorname{Pic}^{0}(X)$, the sum of evaluation maps  
\[
\bigoplus_{p \in U} H^{0}(X, \mathcal{F} \otimes p) \otimes p^{\vee} \to \mathcal{F} 
\]
is surjective. Now,  fix a point $x \in X$ and suppose we are given a surjective map $\mathcal{F} \to \C(x)$. Thus, we have the following commutative diagram:
\[\begin{tikzcd}
\displaystyle \bigoplus_{p \in U} {H^{0}(X, \mathcal{F} \otimes p) \otimes p^{\vee}} \arrow[r]
\arrow{rd}[swap]{\phi} & \mathcal{F} \arrow[d]\\
& \C(x)
\end{tikzcd}
\]
However, the induced map $\phi$ defined by composition is surjective and hence nonzero. In particular, there exists $p \in U$ such that the map $$H^{0}(X, \mathcal{F} \otimes p ) \otimes p^{\vee} \to  \C(x)$$ is nonzero, and hence the map
\[
H^{0}(X, \mathcal{F} \otimes p) \to\C(x)
\]
is nonzero. Thus $\mathcal{F}$ has no essential basepoints as claimed.
\end{proof}
\begin{example}\label{ampleexample}
It is clear that every ample line bundle $L$ on an abelian variety $Y$ satisfies $IT$ with index $0$, and thus is continuously globally generated. Furthermore, by Proposition \ref{haconchenproposition}, we also see that $L$ has no essential basepoints. On the other hand, the converse of Proposition \ref{premierproposition} is not true. Indeed, for an irregular variety $X$,   $\mathcal{O}_{X} \in \operatorname{Pic}^{0}(X)$ has no essential basepoints, however, it is not continuously globally generated.
\end{example}
In the next proposition, we provide a sufficient condition for Definition \ref{continousdefinition} and Definition \ref{essantialdefinition} to be equivalent in the case of line bundles.
\begin{proposition}\label{deuxiemeproposition}
    Let $X$ be an irregular  variety, and let $D$ be a divisor on $X$ such that $h^{0}(X, \mathcal{O}_{X}(D) \otimes p)$ is constant for all $p \in \operatorname{Pic}^{0}(X)$. Then, $\mathcal{O}_{X}(D)$ has no essential basepoints if and only if $\mathcal{O}_{X}(D)$ is continuously globally generated.
\end{proposition}
\begin{proof}
    Assume that $\mathcal{O}_{X}(D)$ has no essential basepoints, then for all $x \in X$ and for any surjective map 
    \[\psi: \mathcal{O}_{X}(D) \to \C(x),\] we can find $p_{x} \in \operatorname{Pic}^{0}(X)$ (Definition \ref{essantialdefinition}) such that the induced map \[H^{0}(X, \mathcal{O}_{X}(D) \otimes p_{x} ) \to H^{0}(X, \C(x))\] 
    is surjective. Thus $$h^{0}(X, \operatorname{Ker}\psi \otimes p_x) = h^{0}(X, \mathcal{O}_{X}(D) \otimes p_{x} ) -1.$$  
    By the upper semi-continuity of $h^{0}(X, \operatorname{Ker}\psi \otimes p)$ as $p$ varies in $\operatorname{Pic}^{0}(X)$, we deduce that 
    \[
    h^{0}(X, \operatorname{Ker}\psi \otimes p_x) = h^{0}(X, \operatorname{Ker}\psi \otimes p)
    \]
    for general $p \in \operatorname{Pic}^{0}(X)$. By assumption, $h^{0}(X, \mathcal{O}_{X}(D) \otimes p)$ is constant for all $p \in \operatorname{Pic}^{0}(X)$. Therefore,  we conclude that the map 
   \begin{equation}\label{surforgen}
       H^{0}(X, \mathcal{O}_{X}(D) \otimes p) \to H^{0}(X, \C(x)) \simeq \C(x)
       \end{equation}
   is surjective for general $p \in \operatorname{Pic}^{0}(X)$. 
   \par Now, we shall establish that $\mathcal{O}_{X}(D)$ is continuously globally generated. In other words, we will prove for all $x \in X$ and for all $U \subseteq \operatorname{Pic}^{0}(X)$ nonempty open subset, the sum of evaluation maps 
\begin{equation}\label{surj1}
\bigoplus_{p \in U} H^{0}(X,\mathcal{O}_{X}(D) \otimes p) \otimes p^{\vee} \to \mathcal{O}_{X}(D)_{|_{x}} \simeq \C(x)
\end{equation}
is surjective. Indeed, the map $(\ref{surforgen})$  is surjective for general $p \in \operatorname{Pic}^{0}(X)$, in particular for some $p_{U} \in U$. Thus, twisting the map (\ref{surforgen}) by $p_{U}^{\vee}$ and then the following map 

\[H^{0}(X, \mathcal{O}_{X}(D) \otimes p_{U}) \otimes p_{U}^{\vee} \to \C(x)
\]
is surjective. Finally, we conclude that the sum of
evaluation maps (\ref{surj1}) is surjective as desired. The converse is proved in Proposition \ref{premierproposition}.
\end{proof}
\begin{example}
As in Example $\ref{ampleexample}$, if $L$ is an ample line bundle on an abelian variety $Y$, then $h^{0}(Y, L \otimes p)$ is constant for all $p \in \operatorname{Pic}^{0}(Y)$ since $\chi(Y, L \otimes p)$ is constant for all $p \in \operatorname{Pic}^{0}(Y)$. Moreover, it has no essential basepoints and is continuously globally generated. Hence, Proposition \ref{deuxiemeproposition} is not empty. 
\par However, for an irregular variety $X$, it can happen that there exists a divisor $D$ such that $h^{0}(X, \mathcal{O}_{X}(D)\otimes p)$ is constant for all $p \in \operatorname{Pic}^{0}(X)$, but $\mathcal{O}_{X}(D)$ is not continuously globally generated. Indeed, for instance, take $X = \mathbb{CP}^{1} \times C$ with $C$ an elliptic curve, and let $p_{2}$ be the second projection. Let $D = K_{X} +B$, where $B$ is a section of $p_{2}$. Thus, $B$ is $p_{2}$-ample and hence, by the relative Kodaira vanishing theorem, $\mathcal{R}^{1}(p_{2})_{*}(\mathcal{O}_{X}(D)) =0$. Also, $(p_{2})_{*}(\mathcal{O}_{X}(D)) = 0$ because $$h^0(\mathbb{CP}^1, \mathcal{O}_{X}(D)_{|_{\mathbb{CP}^1}})= h^0(\mathbb{CP}^1, \mathcal{O}_{\mathbb{CP}^1}(-1)) = 0.$$ 
It follows that $h^{i}(X, \mathcal{O}_{X}(D) \otimes p) = 0$, for all $i$ with $0 \leq i \leq 2$ and for all $p \in \operatorname{Pic}^{0}(X)$. In particular, $\mathcal{O}_{X}(D)$ is not continuously globally generated.
\end{example}

\par We recall the definition of Albanese dimension.
\begin{defn}\label{dimensionalbanese}
    We define the Albanese dimension $\alpha(X)$ of an irregular variety $X$ by 
    \[
    \alpha(X):= \dim \operatorname{alb}(X).
    \]
    Here, $\operatorname{alb}(X)$ is the image of the Albanese map $\operatorname{alb}: X \to \operatorname{alb}(X) \subseteq \operatorname{Alb}(X)$, where $\operatorname{Alb}(X)$ is the Albanese variety.
\end{defn}
\par Consider the Stein factorization $X \xrightarrow{f} Z \xrightarrow{u} \operatorname{alb}(X)$ of $\operatorname{alb}$, where $f$ is a surjective morphism  with connected fibers, we denote a general fiber of $f$ by $F$. Thus, we have the following diagram:
\begin{equation}\label{diagram}
\begin{tikzcd}
X \arrow[r, "f"]
\arrow{rd}[swap]{\operatorname{alb}} & Z \arrow[d,"u"]\\
& \operatorname{alb}(X)
\end{tikzcd}
\end{equation}
where $u: Z \to \operatorname{alb}(X)$ is a finite morphism.

\section{Fujita's Freeness Conjecture on Irregular Varieties}\label{sec3}
Before presenting the main results of this article, we first establish some auxiliary lemmas that will be used in the subsequent proofs. These preliminary results serve as key technical ingredients and help clarify the arguments that follow.
\begin{lemma}\label{lemma1}
Let $X$ be an irregular variety of dimension $n\geq 2$ with $h:X \to Y$ be a morphism to an Abelian variety $Y$. Let $X \xrightarrow{f} Z \xrightarrow{u} Y$ be the Stein factorization of $h$, and let $F$ be a general fiber of $f$. Let $N$ be a divisor on $X$. If $|N_{|_{F}}|$ is basepoint-free and $f_{*}\mathcal{O}_{X}(N)$ is continuously globally generated, then $|N + h^{*}p|$ has no basepoints supported on $F$ for general $p \in \operatorname{Pic}^{0}(Y)$.
\end{lemma}
\begin{proof} Indeed, take a point $x\in F$ such that $z =f(x)$. By assumption, $f_{*}\mathcal{O}_{X}(N)$ is continuously globally generated, that is
for every nonempty open subset $U \subset \operatorname{Pic}^{0}(Z)$, the following direct sum of evaluation maps 
\begin{equation}\label{se1*}
\bigoplus_{p \in U} H^{0}(Z, f_{*}\mathcal{O}_{X}(N)\otimes p) \otimes p^{\vee} \to f_{*}\mathcal{O}_{X}(N)_{|_{z}}
\end{equation}
is surjective. Since $\operatorname{Pic}^{0}(Y) \hookrightarrow \operatorname{Pic}^{0}(Z)$, we have in particular that for every nonempty open subset $U \subset \operatorname{Pic}^{0}(Y)$, the direct sum of evaluation maps
\begin{equation}\label{se1**}
\bigoplus_{p \in U} H^{0}(Z, f_{*}\mathcal{O}_{X}(N)\otimes u^{*}p) \otimes u^{*}p^{\vee} \to f_{*}\mathcal{O}_{X}(N)_{|_{z}}
\end{equation}
is surjective. Furthermore,  \begin{equation}\label{se1}
f_{*}\mathcal{O}_{X}(N)_{|_{z}} = H^{0}(F, \mathcal{O}_{F}(N)) 
\end{equation}
and 
\begin{equation}\label{se2}
H^{0}(Z, f_{*}\mathcal{O}_{X}(N)\otimes u^{*}p) \simeq H^{0}(X, \mathcal{O}_{X}(N) \otimes h^{*}p).
\end{equation}
Combining (\ref{se2}) with (\ref{se1}) and (\ref{se1**}), we obtain that the following sum of maps
\begin{equation}\label{se3}
\bigoplus_{p \in U} H^{0}(X, \mathcal{O}_{X}(N) \otimes h^{*}p) \otimes h^{*}p^{\vee} \to H^{0}(F, \mathcal{O}_{F}(N))
\end{equation}
is surjective. By hypothesis, $|N_{|_{F}}|$ is basepoint-free. Together with (\ref{se3}), this implies that for every nonempty open subset $U \subset \operatorname{Pic}^{0}(Y)$, the following direct sum of evaluation maps
\begin{equation}\label{se4}
\bigoplus_{p \in U} H^{0}(X, \mathcal{O}_{X}(N) \otimes h^{*}p) \otimes h^{*}p^{\vee} \to \mathcal{O}_{X}(N)_{|_{x}}    
\end{equation}
is surjective. In particular, for some $p_{U} \in U$, the map 
\begin{equation}\label{se5}
 H^{0}(X, \mathcal{O}_{X}(N) \otimes h^{*}p_{U}) \otimes h^{*}p_{U}^{\vee} \to \mathcal{O}_{X}(N)_{|_{x}}    
\end{equation}
is surjective. Thus, twisting $(\ref{se5})$ by $h^{*}p_{U}$, we deduce that 
\begin{equation}\label{se6}
 H^{0}(X, \mathcal{O}_{X}(N) \otimes h^{*}p_{U}) \to \mathcal{O}_{X}(N)_{|_{x}}    
\end{equation}
is surjective. Hence, the linear system $$|N + h^{*}p_{U}|$$ has no  basepoints supported on $F$. Now, we define the following subset $S \subseteq \operatorname{Pic}^{0}(Y)$ by 
\[
S := \left\{ p \in \operatorname{Pic}^{0}(Y) \;\middle|\; |N + h^{*}p| \text{ has no basepoints supported on } F \right\}.\]
We claim that \( S \) is dense. Indeed, for any nonempty open subset \( U \subseteq \operatorname{Pic}^{0}(Y) \), one can find an element \( p_{U} \in U \) such that \( |N + h^{*}p_{U}| \) has no basepoints supported on \( F \), as proved. Therefore, \( S \cap U \neq \emptyset \) for every nonempty open subset 
\( U \subseteq \operatorname{Pic}^{0}(Y) \). Hence, \( S \) is dense. 
Otherwise, \( \left(\operatorname{Pic}^{0}(Y) \setminus \overline{S}\right) \cap S \neq \emptyset \), 
a contradiction. It follows that 
\[
|N + h^{*}p|
\]
has no basepoints supported on \( F \) for general 
\( p \in \operatorname{Pic}^{0}(Y) \).
\end{proof}
\begin{lemma}\label{lemma2}
Let $Y$ and $Z$ be irregular varieties and  $u: Z \to Y$ be a finite morphism. Let $\mathcal{F}$ be a coherent sheaf on $Z$.
\begin{itemize}
    \item [1)] If $H^{1}(Y, u_{*}\mathcal{F} \otimes p) =0$ for all $p \in \operatorname{Pic}^{0}(Y)$ and $u_{*}\mathcal{F}$ has no essential basepoints, then $\mathcal{F}$ has no essential basepoints. \item[2)]  If $u_{*}\mathcal{F}$ is continuously globally generated, then $\mathcal{F}$ is continuously globally generated.
    \end{itemize}
\end{lemma}
\begin{proof}
Assuming that  $H^{1}(Y, u_{*}\mathcal{F} \otimes p) =0$ for all $p \in \operatorname{Pic}^{0}(Y)$ and that $u_{*}\mathcal{F}$ has no essential basepoints, we will prove that $\mathcal{F}$ has no essential basepoints. Indeed,
we fix $z \in Z$, $y = u(z) \in Y$, and  assume that there exists the following exact sequence:
\begin{equation}\label{seq1}
0 \to \mathcal{K} \to \mathcal{F} \to \C(z) \to 0 \text{.}
\end{equation}
By applying $u_{*}$ to this exact sequence, we obtain: 
\begin{equation}\label{seq2}
0 \to u_{*}\mathcal{K} \to u_{*}\mathcal{F} \to u_{*}\C(z) \to \mathcal{R}^{1}u_{*}\mathcal{K} \to \cdots
\end{equation}
Since the map $u$ is finite, we have $\mathcal{R}^{1}u_{*}\mathcal{K} = 0$. Moreover, it is clear that $u_{*}\C(z) = \C(y)$, so the previous sequence reduces to a short exact sequence of sheaves:
\begin{equation}\label{seq3}
0 \to u_{*}\mathcal{K} \to u_{*}\mathcal{F} \to \C(y) \to 0 \text{.}
\end{equation}
Since  $u_{*}\mathcal{F}$ has no essential basepoints, there exists $p_{y} \in \operatorname{Pic}^{0}(Y)$ such that the map 
\begin{equation}\label{seq4}
H^{0}(Y, u_{*}\mathcal{F} \otimes p_{y}) \to \C(y)  
\end{equation}
is surjective. Twisting the sequence (\ref{seq3}) by $p_{y}$ and  passing to the associated long exact sequence in cohomology, we obtain:
\[
    0 \to H^{0}(Y, u_{*}\mathcal{K} \otimes p_{y}) \to H^{0}(Y, u_{*}\mathcal{F} \otimes p_{y}) \to \C(y)\]
    \[\to H^{1}(Y, u_{*}\mathcal{K} \otimes p_{y}) \to H^{1}(Y, u_{*}\mathcal{F} \otimes p_{y})
\to 0 \text{.} \]
From the assumption that $$H^{1}(Y, u_{*}\mathcal{F} \otimes p_{y}) =0$$ and the surjectivity of the map in (\ref{seq4}), it follows that $$H^{1}(Y, u_{*}\mathcal{K} \otimes p_{y}) = 0 \text{,}$$ which in turn implies $$H^{1}(Z, \mathcal{K} \otimes u^{*}p_{y}) = 0 \text{.}$$ 
We now twist the sequence (\ref{seq1}) by $u^{*}p_{y}$ and pass to the associated long exact sequence in cohomology, obtaining:
\[
0 \to H^{0}(Z, \mathcal{K} \otimes u^{*}p_{y}) \to H^{0}(Z, \mathcal{F} \otimes u^{*}p_{y}) \to \C(z) \to 0 \text{.}
\]
It follows that the map $$H^{0}(Z, \mathcal{F} \otimes u^{*}p_{y}) \to \C(z)$$
is surjective. We conclude that $\mathcal{F}$ has no essential basepoints.    
\par Now, assume that  $u_{*}\mathcal{F}$ is continuously globally generated. That is, for every $y \in Y$ and for every nonempty open subset $U \subset \operatorname{Pic}^{0}(Y)$, the following direct sum of evaluation maps 
\begin{equation}\label{seq5}
\bigoplus_{p \in U} H^{0}(Y, u_{*}\mathcal{F}\otimes p) \otimes p^{\vee} \to u_{*}\mathcal{F}_{|_{y}}
\end{equation}
is surjective. Our goal is to prove that $\mathcal{F}$ is continuously globally generated. Indeed, consider a nonzero element $v \in \mathcal{F}_{|_{z}} \simeq u_{*}(\mathcal{F}_{|_{z}})_{|_{y}}$ where the isomorphism holds because $u$ is finite. This implies that $v \in u_{*}\mathcal{F}_{|_{y}}$. Then, by the surjectivity of the map in (\ref{seq5}), we can find  elements $(p_{i})_{1 \leq i \leq t}$ in $\operatorname{Pic}^{0}(Y)$ and section $(s_{i})_{1\leq i \leq t}$ in $H^{0}(Y, u_{*}\mathcal{F} \otimes p_{i})$ such that $s_{i}(y) \neq 0$ for all $i$ with  $1 \leq i \leq t$. By the isomorphism $$H^{0}(Y, u_{*}\mathcal{F}\otimes p_{i}) \simeq H^{0}(Z, \mathcal{F} \otimes u^{*}p_{i}) $$
which holds for all $i$ with $1 \leq i \leq t$, we can find nonzero sections  $(u_{i})_{1\leq i \leq t}$  of $H^{0}(Z, \mathcal{F} \otimes u^{*}p_{i})$
such that $v$ has a preimage under the following sum of evaluation map  defined by $(u_{i})_{1\leq i \leq t}$:
\begin{equation}\label{seq6}
\bigoplus_{p \in U} H^{0}(Z, \mathcal{F}\otimes u^{*}p) \otimes u^{*}p^{\vee} \to \mathcal{F}_{|_{z}} \text{.}
\end{equation}
Thus, we conclude that $\mathcal{F}$ is continuously globally generated.
\end{proof}
\begin{lemma}\label{lemma3}
Let $X$ be an irregular variety, and $h:X \to Y$ be a morphism to an Abelian variety $Y$. Let $D$ be a nef and big divisor on $X$. If $h_{*}\mathcal{O}_{X}(K_{X} +D) \neq 0$, then $h_{*}\mathcal{O}_{X}(K_{X} +D)$ satisfies $IT$ with index $0$.
\end{lemma}
\begin{proof}
Applying Kawamata-Viehweg's vanishing theorem (\cite[Theorem 1]{kawamata-viehweg}, \cite[Corollary 0.3]{viehwegvanishing}), we have  $$H^{i}(X, \mathcal{O}_{X}(K_{X} +D) \otimes h^{*}p) = 0 \hspace{0.2cm} \text{for all } i\geq 1 \hspace{0.2cm} \text{and for all }  p \in \operatorname{Pic}^{0}(Y) \text{.}$$
Furthermore, by the relative Kawamata-Viehweg's vanishing theorem (\cite[Theorem 2.2.1]{BMSA}): $$\mathcal{R}^{i}h_{*}(\mathcal{O}_{X}(K_{X} + D) \otimes h^{*}p) = 0 \hspace{0.2cm} \text{for all } i \geq 1 \hspace{0.2cm} \text{and for all }  p \in \operatorname{Pic}^{0}(Y).$$ Thus, it follows from  Leray's spectral sequence argument that \[
H^{i}(Y, h_{*}\mathcal{O}_{X}(K_{X} + D) \otimes p) = 0 \hspace{0.2cm} \text{for all }  i\geq 1 \hspace{0.2cm} \text{and for all }  p \in  \operatorname{Pic}^{0}(Y)\text{.}
\]
Consequently, $h_{*}\mathcal{O}_{X}(K_{X} +D)$ satisfies $IT$ with index $0$.
\end{proof}
\begin{proposition}\label{newpropo}
Let $X$ be an irregular variety, and $h:X \to Y$ be a morphism to an Abelian variety $Y$. Let $N$ and $M$ be divisors on $X$. If a point $x \in X$ is not a basepoint of $|N + h^{*}p|$ for  general $p \in \operatorname{Pic}^{0}(Y)$,
and $x$ is not a basepoint of $|M + h^{*}p|$ for  general $p \in  \operatorname{Pic}^{0}(Y)$, then $x$ is not a base point of $|N +M|$.
\end{proposition}
\begin{proof}
By assumption, \( x \in X \) is not a base point of \( |M + h^{*}p| \) for general \( p \in \operatorname{Pic}^{0}(Y) \). Since the map \( p \mapsto p^{\vee} \) is an automorphism of \( \operatorname{Pic}^{0}(Y) \), the image of a general point under this map is again a general point. Therefore, \( x \) is not a base point of \( |M + h^{*}p^{\vee}| \) for general \( p \in \operatorname{Pic}^{0}(Y) \). Hence, \( x \) is not a base point of \( |N + M| \).
\end{proof}
\begin{theorem}\label{maintheorem1}
Let $X$ be an irregular variety of dimension $n\geq 2$ with Albanese dimension $1 \leq \alpha(X)<n$. Let $X \xrightarrow{f} Z \xrightarrow{u} \operatorname{alb}(X) \subseteq \operatorname{Alb}(X)$ be the Stein factorization of the Albanese morphism $\operatorname{alb}$, and let $F$ be a general fiber of the morphism $f$. Let $D$ be an ample divisor on $X$. 
\par If  Conjecture \ref{fujitaconjecture} holds in dimension $<n$, then $|K_{X} + mD + \operatorname{alb}^{*}p|$ has no  basepoint supported on $F$ for all $m \geq n-\alpha(X)+1$ and for general $p \in \operatorname{Pic}^{0}(\operatorname{Alb}(X))$.
\par Additionally, if the following condition is satisfied:
\begin{itemize}
    \item[$(*)$]There exists an integer $r$ with $1\leq r \leq \alpha(X)$ such that $|rD_{|_{F}}|$ is basepoint-free and $rD - K_{X}$ is nef and big,
\end{itemize}
then $|K_{X} + mD|$ has no  basepoint supported on $F$ for all $m \geq n+1$.
\end{theorem}
\begin{proof}
From the hypothesis, Conjecture \ref{fujitaconjecture} holds for lower-dimensional varieties, meaning that $|K_{F}+ mD_{|_{F}}|$ is basepoint-free for all $m \geq n - \alpha(X) +1$. In particular,
$$h^{0}(F, \omega_{F} \otimes \mathcal{O}_{F}(mD)) \neq 0 \hspace{0.2cm} \text{for all }  m \geq n-\alpha(X)+1 \text{,}$$
since $(K_{X} +mD)_{|_{F}} = K_{F} + mD_{|_{F}}$.
Thus, $f_{*}\mathcal{O}_{X}(K_{X} + mD)$ is a nonzero coherent sheaf on $Z$ for all $m \geq n- \alpha(X)+1$. Therefore, $\operatorname{alb}_{*}\mathcal{O}_{X}(K_{X} +mD)$ is a nonzero coherent sheaf for all $m \geq n-\alpha(X)+1$, as the map $u$ is finite. 
\par Applying Lemma \ref{lemma3}, it follows that  $\operatorname{alb}_{*}\mathcal{O}_{X}(K_{X} +mD)$ satisfies $IT$ with index $0$. By Remark \ref{remarkproperties} and Proposition \ref{propoMregularity}, we deduce that $\operatorname{alb}_{*}\mathcal{O}_{X}(K_{X} +mD)$ is continuously globally generated.
\\Applying Lemma \ref{lemma2}, we conclude that $f_{*}\mathcal{O}_{X}(K_{X} + mD)$ is continuously globally generated. Again, by assumption, we know that $|K_{F}+ mD_{|_{F}}|$ is basepoint-free for all $m \geq n - \alpha(X) +1$. Then, applying Lemma \ref{lemma1}, we deduce that  $|K_{X} + mD + \operatorname{alb}^{*}p|$ has no  basepoint supported on $F$ for all $m \geq n-\alpha(X)+1$ and for general $p \in \operatorname{Pic}^{0}(\operatorname{Alb}(X))$. This proves the first statement.
\par Now, suppose that condition $(*)$ is satisfied. In particular, $|rD_{|_{F}}|$ is basepoint-free, and 
\[
h^{0}(F, \mathcal{O}_{F}(rD)) \neq 0.
\]
Therefore, $f_{*}\mathcal{O}_{X}(rD)$ is a nonzero sheaf on $Z$, and so is $\operatorname{alb}_{*}\mathcal{O}_{X}(rD)$. We  observe that
\[
rD = K_{X} + rD -K_{X}.
\]
Thus, applying Lemma \ref{lemma3}, we obtain that  $\operatorname{alb}_{*}\mathcal{O}_{X}(rD)$ satisfies $IT$ with index $0$, and therefore,  $\operatorname{alb}_{*}\mathcal{O}_{X}(rD)$ is continuously globally generated. Applying Lemma \ref{lemma2}, we deduce that $f_{*}\mathcal{O}_{X}(D)$ is continuously globally generated.
\\Again, by assumption, we know that $|rD_{|_{F}}|$ is basepoint-free. Then, applying Lemma \ref{lemma1}, we conclude that 
  $$|rD + \operatorname{alb}^{*}p|$$ has no basepoint supported on a general fiber $F$  for general $p \in \operatorname{Pic}^{0}(\operatorname{Alb}(X))$.
\par Finally, applying Proposition \ref{newpropo} to $|K_{X} + mD + \operatorname{alb}^{*}p|$ and $|rD + \operatorname{alb}^{*}p|$, we deduce that  $$|K_{X} + m^{'}D|$$ has no  basepoints supported on $F$ for all $m^{'} \geq n + r-\alpha(X)+1$, and in particular, for all $m^{'} \geq n+1$. 
\end{proof}
\begin{remark}
In Theorem \ref{maintheorem1}, the condition that $rD - K_{X}$ is nef and big for some $r$ with $1 \leq r \leq \alpha(X)$  is always satisfied if  $X$ is an irregular variety with a nef anticanonical bundle.
\end{remark}
\par The following basic example in dimension $2$ illustrates that  Condition $(*)$ in Theorem \ref{maintheorem1} can be weaker than requiring $D$ to be globally generated. 
\begin{example}\label{simpleexample}
Let $\mathcal{E}$ be a rank $2$-vector bundle on elliptic curve $C$ given by the following non-split short exact sequence:
\begin{equation}\label{exactseq}
0 \to \mathcal{O}_{C} \to \mathcal{E} \to \mathcal{O}_{C} \to 0  \text{.}
\end{equation}
We take the ruled surface \( X = \mathbb{P}(\mathcal{E}) \xrightarrow{f} C \), where $F$ is a fiber. We know that the canonical line bundle \( \omega_{X} = \mathcal{O}_{X}(-2) \), which means \( \omega_{X}^{\vee} = \mathcal{O}_{X}(2) \) is nef. Additionally, as another feature of this example, it turns out that the tangent bundle $T_{X}$ is nef. Furthermore, the irregularity \( q(X) = \alpha(X) = 1 \) since \( h^{1}(X, \mathcal{O}_{X}) = h^{1}(C, \mathcal{O}_{C}) = 1 \). Let \( D \) be a divisor such that \( \mathcal{O}_{X}(D) := \mathcal{O}_{X}(1) \otimes \mathcal{O}_{X}(F) \), and we observe that \( D \) is ample. However, we claim that it is not basepoint-free.
\begin{proof}[Proof of the claim]
Since \( H^{0}(X, \mathcal{O}_{X}(D)) = H^{0}(C, f_{*}\mathcal{O}_{X}(D)) = H^{0}(C, \mathcal{E} \otimes \mathcal{O}_{C}(x)) \) (where \( x \) is a point on \( C \)), and by twisting the short exact sequence (\ref{exactseq}) by \( \mathcal{O}_{C}(x) \), we obtain the following long exact sequence of cohomology:
\[
    0 \to H^{0}(C, \mathcal{O}_{C}(x)) \to H^{0}(C, \mathcal{E} \otimes \mathcal{O}_{C}(x)) \to H^{0}(C, \mathcal{O}_{C}(x)) \to H^{1}(C, \mathcal{O}_{C}(x)) \to \dots
\]
But since \( H^{1}(C, \mathcal{O}_{C}(x)) = 0 \), we conclude that \( h^{0}(C, \mathcal{E} \otimes \mathcal{O}_{C}(x)) = 2 \). Now, let \( B \) be a divisor such that \( \mathcal{O}_{X}(B) = \mathcal{O}_{X}(1) \). Then, we have the following short exact sequence:
\[
0 \to \mathcal{O}_{X}(F) \to \mathcal{O}_{X}(D) \to \mathcal{O}_{B}(D) \to 0,
\]
which induces the following long exact sequence on cohomology groups:
\begin{equation}\label{exactseq1}
0 \to H^{0}(X, \mathcal{O}_{X}(F)) \to H^{0}(X, \mathcal{O}_{X}(D)) \to H^{0}(B, \mathcal{O}_{B}(D)) \to H^{1}(X, \mathcal{O}_{X}(F)) \to \dots
\end{equation}
By the Kodaira vanishing theorem, we see  that $$h^{1}(X, \mathcal{O}_{X}(F)) = h^{1}(X, \mathcal{O}_{X}(K_{X} + 2B + F)) = 0$$ since $2B + F$ is ample. Also, we have $$h^{0}(X, \mathcal{O}_{X}(F)) = h^{0}(B, \mathcal{O}_{B}(D)) = 1.$$ The previous sequence $(\ref{exactseq1})$ becomes a short exact sequence:
\[
0 \to H^{0}(X, \mathcal{O}_{X}(F)) \to H^{0}(X, \mathcal{O}_{X}(D)) \to H^{0}(B, \mathcal{O}_{B}(D)) \to 0
\]
Finally, we fix two linearly independent sections of \( H^{0}(X, \mathcal{O}_{X}(D)) \) and restrict them to \( B \). These sections induce a nonzero section \( s \), which vanishes on $D_{|_{B}}$. Therefore, $D$ is not basepoint-free.
\end{proof}
We remark that $D$ satisfies the condition $(*)$ in Theorem \ref{maintheorem1}. Indeed, $\mathcal{O}_{F}(\alpha(X)D) = \mathcal{O}_{\mathbb{CP}^{1}}(1)$, which is ample and globally generated. Also, $$\mathcal{O}_{X}(\alpha(X)D - K_{X}) = \mathcal{O}_{X}(3B + F)$$ is ample on $X$. This example shows that there exists a divisor $D$ satisfying condition $(*)$ in Theorem \ref{maintheorem1} and not necessarily globally generated. Furthermore, $$K_{X} + mD = (m-2)B + mF$$  is globally generated, for all $m \geq 3$. 
\end{example}
\begin{corollary}\label{mainresult2}
Let $X$ be an irregular variety of dimension $n\geq 3$ with Albanese dimension $1 \leq \alpha(X) < n$ and $K_{X}$ is ample. Let $X \xrightarrow{f} Z \xrightarrow{u} \operatorname{alb}(X) \subseteq \operatorname{Alb}(X)$ be the Stein factorization of the Albanese morphism $\operatorname{alb}$, and let $F$ be a general fiber of the morphism $f$. If both $|mK_{F}|$ and $|rK_{F}|$ are basepoint-free for all $m \geq n + 2 -\alpha(X)$ and for some $r$ with $1 < r \leq \alpha(X)$, then $|mK_{X}|$ has no  basepoints supported on $F$ for all $m \geq n+2$.  
\end{corollary}
\begin{proof}
Apply  Theorem \ref{maintheorem1} for $D = K_{X}$.
\end{proof}
\begin{remark}
    If $r \geq n+2-\alpha(X)$ in Corollary \ref{mainresult2}, then it suffices to assume the basepoint-freeness of $|mK_{F}|$ for all $m \geq n + 2 - \alpha(X)$.
\end{remark}
\par An interesting situation arises when the Albanese map is a  locally trivial fibration. In this case, we can simplify the setting of Theorem \ref{maintheorem1}. For instance, the Stein factorization is trivial, and  under the assumption of condition $(*)$ in Theorem \ref{maintheorem1}, we conclude that the linear system  $|K_{X} + mD|$ is basepoint-free for all $m \geq n+1$. As an example of this  situation, we state the following theorem.
\begin{theorem}
    \label{supplement}
Let $X$ be an irregular variety of dimension $n \geq 2$ with $-K_{X}$ nef. Let $\operatorname{alb}: X \to \operatorname{Alb}(X)$ be the Albanese map, and let $D$ be an ample divisor on $X$. If Conjecture \ref{fujitaconjecture} holds in dimension $<n$ and there exists an integer $r$ with $1\leq r \leq \alpha(X)$ such that $|rD_{|_{F}}|$ is basepoint-free for every fiber $F$ of $\operatorname{alb}$, then Conjecture \ref{fujitaconjecture} holds for $X$.
\end{theorem}
\begin{proof}
    If $-K_{X}$ is nef, then, by \cite[Theorem 1.2]{Cao}, the Albanese map $\operatorname{alb}$ is a  locally trivial fibration. Hence, all the fibers are isomorphic.
\par Note that we may assume the fibers of  $\operatorname{alb}$ have positive dimension, that is, $\alpha(X) < n$. Indeed, if $-K_{X}$ is nef and $\alpha(X) = n$ for a variety $X$, then $X$ is an abelian variety. In that case, $K_{X} =0$ and $|2D|$ is basepoint-free.
\par By assumption, Conjecture \ref{fujitaconjecture} holds in low-dimensions. Then, the linear system $$|K_{F} +mD_{|_{F}}|$$ is basepoint-free for all $m \geq n-\alpha(X)+1$ and for every fiber $F$. In particular,
$$h^{0}(F, \omega_{F} \otimes \mathcal{O}_{F}(mD)) \neq 0 \hspace{0.2cm} \text{for all }  m \geq n-\alpha(X)+1 \text{.}$$
Thus, $\operatorname{alb}_{*}\mathcal{O}_{X}(K_{X} +mD)$ is a nonzero coherent sheaf for all $m \geq n-\alpha(X)+1$. 
\\Furthermore, $\operatorname{alb}_{*}\mathcal{O}_{X}(K_{X} + mD)$ satisfies $IT$ with index $0$, by Lemma \ref{lemma3}. Consequently, by Remark \ref{remarkproperties} and Proposition \ref{propoMregularity}, it is continuously globally generated. Applying Lemma \ref{lemma1}, we deduce that 
 $$|K_{X} + mD + \operatorname{alb}^{*}p|$$ is basepoint-free for general $p \in \operatorname{Pic}^{0}(\operatorname{Alb}(X))$ and for all $m \geq n-\alpha(X)+1$. 
\par Since $-K_{X}$ is nef by assumption, the divisor $rD-K_{X}$ is ample. Moreover, by assumption, $|rD_{|_{F}}|$ is basepoint-free for some $r$ with $1\leq r \leq \alpha(X)$. By the base change theorem, this implies that $\operatorname{alb}_{*}\mathcal{O}_{X}(rD)$ is a nonzero locally free sheaf. Furthermore, it satisfies $IT$ with index $0$, by Lemma \ref{lemma3}. 
\\Therefore, by Remark \ref{remarkproperties} and Proposition \ref{propoMregularity}, $\operatorname{alb}_{*}\mathcal{O}_{X}(rD)$ is continuously globally generated. Applying Lemma \ref{lemma1} again, we deduce that $$|rD + \operatorname{alb}^{*}p|$$ is basepoint-free for general $p \in \operatorname{Pic}^{0}(\operatorname{Alb}(X))$. Thus, applying Proposition \ref{newpropo} to $|K_{X} + mD + \operatorname{alb}^{*}p|$ and $|rD + \operatorname{alb}^{*}p|$, we conclude that Conjecture \ref{fujitaconjecture} holds for $X$.
\end{proof}
\begin{example}\label{6dimensionexample}
    Since Fujita's Freeness Conjecture \ref{fujitaconjecture} holds for varieties of dimension $5$ by \cite{feizhu}, then by applying Theorem \ref{supplement}, it follows that the conjecture holds for irregular varieties of dimension $6$ with $-K_{X}$ nef, provided that  $|rD_{|_{F}}|$ is basepoint-free for some $r$ satisfying  $1\leq r  \leq \alpha(X)$. Here, $F$ denotes a fiber of $\operatorname{alb}$.
\end{example}
\begin{example}
    More generally than in Example \ref{6dimensionexample}, if $X$ is an irregular variety of dimension $n$ with  $-K_{X}$ nef, $\alpha(X) \geq n-5$, and if $|rD_{|_{F}}|$ is basepoint-free for some $r$ such that  $1\leq r  \leq \alpha(X)$, then Conjecture \ref{fujitaconjecture} holds for $X$.
\end{example}
\begin{example}
Example \ref{simpleexample} is also an instance of Theorem \ref{supplement}. 
\end{example}
\begin{remark}
In Theorem \ref{supplement}, the condition that $|rD_{|_{F}}|$ is basepoint-free for some $r$ satisfying  $1\leq r  \leq \alpha(X)$ is equivalent  to the induction hypothesis if $\alpha(X) \geq \frac{n+1}{2}$ and the fiber $F$ of $\operatorname{alb}: X \to \operatorname{Alb(X)}$ is a $K$-trivial variety ($K_{F}=0$). Indeed, take $r := n- \alpha(X) +1$, and  note that $$(n -\alpha(X)+1)D_{|_{F}} = K_{F} + (n-\alpha(X)+1)(D_{|_{F}}) \text{.}$$
\end{remark}
\section{Basepoint-Freeness of
Adjoint Series for Varieties Fibered over Abelian Varieties}
\par In the following proposition, we consider the base to be an abelian variety, not necessarily the Albanese variety.  We are interested in studying the linear system defined by the canonical sheaf twisted by an ample line bundle from the base, under the assumption that the morphism is an algebraic fiber space and that a general fiber has a basepoint-free canonical bundle.
\begin{proposition}\label{kollarvanishing}
    Let $h: X \to Y$ be a surjective morphism with connected fibers onto an abelian variety $Y$ of dimension $g$, and let $F$ be a general fiber of $h$. If $|K_{F}|$ is basepoint-free and $\Theta$ is an ample divisor on $Y$, then $|K_{X} + 2 h^{*}\Theta|$ has no basepoints supported on $F$.
\end{proposition}
\begin{proof}
From the hypothesis, $|K_{F}|$ is basepoint-free. In particular,
$$h^{0}(F, \omega_{F}) \neq 0\text{.}$$
Thus, $h_{*}\mathcal{O}_{X}(K_{X})$ is a nonzero coherent sheaf.
   \par  By Kollár's vanishing theorem (\cite[Theorem 2.1]{kollarhigherdirect}), we have $$H^{i}(Y, h_{*}\mathcal{O}_{X}(K_{X}) \otimes \mathcal{O}_{Y}(\Theta) \otimes p) = 0 \hspace{0.2cm} \text{for all }  i \geq 1 \hspace{0.2cm} 
    \text{and for all }  p \in \operatorname{Pic}^{0}(Y). $$ Thus, $h_{*}\mathcal{O}_{X}(K_{X}) \otimes \mathcal{O}_{Y}(\Theta)$ is a nonzero coherent sheaf satisfies $IT$ with index $0$. Therefore, it is continuously globally generated. By assumption, $|K_{F}|$ is basepoint-free. Thus, applying Lemma \ref{lemma1}, we deduce that $$|K_{X} + h^{*}\Theta + h^{*}p|$$ has no basepoint supported on $F$ for general $p \in \operatorname{Pic}^{0}(Y)$. 
\par Moreover, since $\Theta$ is an ample divisor on $Y$, it satisfies $IT$ with index $0$. Furthermore, $h^{*}\Theta$ is trivial when  restricted to any fiber. Thus, applying Lemma \ref{lemma1} again, we deduce that  $$|h^{*}(\Theta + p)|$$ is basepoint-free for general $p \in \operatorname{Pic}^{0}(Y)$. 
\par Finally, applying Proposition \ref{newpropo} to $|K_{X} + h^{*}\Theta + h^{*}p|$ and $|h^{*}(\Theta + p)|$, we conclude that $$|K_{X} + 2 h^{*}\Theta|$$ has no basepoints supported on $F$.
\end{proof}
\par In the next theorem, we consider the Albanese map and  remove  condition $(*)$ from Theorem \ref{maintheorem1}. We observe that if we assume the adjoint linear system of a general fiber is basepoint-free, then the linear system defined by the adjoint canonical bundle twisted with an ample line bundle from the base has no basepoints supported on a general fiber.

\begin{theorem}\label{supplement3}
    Assume that the Albanese map $\operatorname{alb}: X \to \operatorname{Alb}(X)$ is a surjective morphism with connected fibers, and let $F$ be a general fiber. Let $D$ be a nef and big divisor on $X$, and $\Theta$ an ample divisor on $\operatorname{Alb}(X)$. If there exists an integer $c > 0$  such that $|K_{F} + mD_{|_{F}}|$ is basepoint-free on $F$ for all $m \geq c$, then $|K_{X} + mD + \operatorname{alb}^{*}\Theta|$ has no basepoints supported on a general fiber $F$ for all $m \geq c$. 
\end{theorem}
\begin{proof}
From the hypothesis, $|K_{F}+ mD_{|_{F}}|$ is basepoint-free for all $m \geq c$. In particular,
$$h^{0}(F, \omega_{F} \otimes \mathcal{O}_{F}(mD)) \neq 0 \hspace{0.2cm} \text{for all }  m \geq c \text{.}$$
Thus, $\operatorname{alb}_{*}\mathcal{O}_{X}(K_{X} +mD)$ is a nonzero coherent sheaf for all $m \geq c$. 
\par Applying Lemma \ref{lemma3}, it follows that  $\operatorname{alb}_{*}\mathcal{O}_{X}(K_{X} +mD)$ satisfies $IT$ with index $0$. By Remark \ref{remarkproperties} and Proposition \ref{propoMregularity}, we deduce that $\operatorname{alb}_{*}\mathcal{O}_{X}(K_{X} +mD)$ is continuously globally generated. Thus, applying Lemma \ref{lemma1}, we conclude that $$|K_{X} + mD + \operatorname{alb}^{*}p|$$ has no  basepoint supported on $F$ for all $m \geq c$ and for general $p \in \operatorname{Pic}^{0}(\operatorname{Alb}(X))$. 
\par Applying Lemma \ref{lemma1} to $\operatorname{alb}^{*}(\Theta)$, we obtain that $$|\operatorname{alb}^{*}(\Theta + p)|$$ is basepoint-free for general $p \in \operatorname{Pic}^{0}(\operatorname{Alb}(X))$. 
\par Finally, applying Proposition \ref{newpropo}, we conclude that
$$|K_{X} + mD + \operatorname{alb}^{*}\Theta|$$ has no  basepoints supported on $F$ for all $m \geq c$.
\end{proof}
\section{Basepoint-Freeness of Adjoint Series for Varieties of Maximal Albanese Dimension}
In what follows, we present some results on the basepoint-freeness of linear series, assuming that $X$ is an irregular variety of dimension $n$ with maximal Albanese dimension, that is, $\alpha(X) = n$. Related results on basepoint-freeness, in the setting of varieties whose Albanese morphism is finite, were obtained in \cite[Theorem 5.1]{popapareschi1}. First, we recall the definition of the exceptional set of $\operatorname{alb}$.
\begin{defn}
Let $X$ be an irregular variety of maximal Albanese dimension. Let $X \xrightarrow{f} Z \xrightarrow{u} \operatorname{alb}(X) \subseteq \operatorname{Alb}(X)$ be the Stein factorization of $\operatorname{alb}$. The exceptional set of $\operatorname{alb}$ and $f$ is the subset of $X$ where the morphism $\operatorname{alb}$ is not finite.
\end{defn}
\begin{theorem}\label{mainresult3}
    Let $X$ be an irregular  variety of maximal Albanese dimension, that is $\alpha(X) = n$, and let $D$ be a nef and big divisor on $X$ such that $nD - K_{X}$ is nef and big, or $nD$ is continuously globally generated. Then $|K_{X} + mD|$ is basepoint-free outside the exceptional set of $\operatorname{alb}$ for all $m \geq n+1$. 
\end{theorem}
\begin{proof}Since $X$ is a variety of maximal Albanese dimension, the Albanese map is generically finite. We take the Stein factorization $X \xrightarrow{f} Z \xrightarrow{u} \operatorname{alb}(X) \subseteq \operatorname{Alb}(X)$ of $\operatorname{alb}$, where $f$ is a birational map. Thus, $f_{*}\mathcal{O}_{X}(K_{X} +D)$ is nonzero, and consequently, so is $\operatorname{alb}_{*}\mathcal{O}_{X}(K_{X} + D)$. Thus, $\operatorname{alb}_{*}\mathcal{O}_{X}(K_{X} + D)$ is a nonzero coherent sheaf that satisfies $IT$ with index $0$, by Lemma \ref{lemma3}. Therefore, it is continuously globally generated. By Lemma \ref{lemma2}, this implies that $$|K_{X} + D + \operatorname{alb}^{*}p|$$ has no basepoints outside the exceptional set of $\operatorname{alb}$ for general $p \in \operatorname{Pic}^{0}(\operatorname{Alb}(X))$.

\par By assumption $nD - K_{X}$ is nef and big. Then,  applying Lemma \ref{lemma3}, we deduce that $\operatorname{alb}_{*}\mathcal{O}_{X}(mD)$ is a nonzero coherent sheaf that satisfies  $IT$ with index $0$ for all $m \geq n$. Hence, it is continuously globally generated. Thus, $$|mD + \operatorname{alb}^{*}p|$$ is basepoint-free outside the exceptional set of $\operatorname{alb}$ for all $p \in \operatorname{Pic}^{0}(\operatorname{Alb}(X))$ and for all $m \geq n$, by Lemma \ref{lemma2}. 
\par Finally, applying Proposition \ref{newpropo}, we conclude that $|K_{X} + m^{'}D|$ is basepoint-free outside the exceptional set of $\operatorname{alb}$ for all $m^{'} \geq n+1$. 
\end{proof}
\begin{remark}
    If we assume that the Albanese map is finite in Theorem \ref{mainresult3}, under the same conditions on $D$ and $nD - K_{X}$, then $|K_{X} + mD|$ is basepoint-free on $X$ for all $m \geq n+1$.  
\end{remark}
\begin{corollary}\label{mainresult4}
Let $X$ be a minimal variety of general type with maximal Albanese dimension, $\alpha(X) = n$. Then, $|4K_{X}|$ is basepoint-free outside the exceptional set of $\operatorname{alb}$. 
\end{corollary}
\begin{proof}
We take the Stein factorization $X \xrightarrow{f} Z \xrightarrow{u} \operatorname{alb}(X) \subseteq \operatorname{Alb}(X)$ of $\operatorname{alb}$, where $f$ is a birational map. Thus, $f_{*}\mathcal{O}_{X}(2K_{X})$ is nonzero, and consequently, so is $\operatorname{alb}_{*}\mathcal{O}_{X}(2K_{X})$.
\par Since $K_{X}$ is nef and big, we apply Lemma \ref{lemma3} to obtain that
$\operatorname{alb}_{*}\mathcal{O}_{X}(2K_{X})$ is a nonzero coherent sheaf that satisfies \emph{IT} with index $0$. Consequently, it is continuously globally generated. Thus,
\[
|2K_{X} + \operatorname{alb}^{*}p|
\]
is basepoint-free outside the exceptional set of $\operatorname{alb}$ for all $p \in \operatorname{Pic}^{0}(\operatorname{Alb}(X))$. Using this fact twice, we conclude that  $|4K_{X}|$ is basepoint-free outside the exceptional set of $\operatorname{alb}$ by Proposition \ref{newpropo}.
\end{proof}
\begin{remark}
This last corollary is not sharp. Indeed, in \cite{trio}, the authors proved that $|3K_{X}|$ is birational under the assumption that $X$ is an irregular variety of general type and of maximal Albanese dimension. 
\end{remark}
The next proposition is an analogue of Proposition \ref{kollarvanishing} for varieties admitting a finite morphism to an abelian variety.
\begin{proposition}
   Let $h: X \to Y$ be a surjective finite morphism  from $X$ to an abelian variety $Y$. Let $\Theta$ be an ample divisor on $Y$. Then $|K_{X} + 2 h^{*}(\Theta)|$ is basepoint-free on $X$. 
\end{proposition}
\begin{proof}
 \par  By Kollár's vanishing theorem (\cite[Theorem 2.1]{kollarhigherdirect}), we have $$H^{i}(Y, h_{*}\mathcal{O}_{X}(K_{X}) \otimes \mathcal{O}_{Y}(\Theta) \otimes p) = 0 \hspace{0.2cm} \text{for all }  i \geq 1 \hspace{0.2cm} 
    \text{and for all }  p \in \operatorname{Pic}^{0}(Y). $$ Thus, $h_{*}\mathcal{O}_{X}(K_{X}) \otimes \mathcal{O}_{Y}(\Theta)$ is a nonzero coherent sheaf that satisfies $IT$ with index $0$. Therefore, it is continuously globally generated. By Lemma \ref{lemma2}, it follows that
    $$|K_{X} + h^{*}\Theta + h^{*}p|$$ is basepoint-free for general $p \in \operatorname{Pic}^{0}(Y)$. 
\par Since $\Theta$ is an ample divisor on $Y$, it satisfies $IT$ with index $0$.  Thus,  $$|h^{*}(\Theta + p)|$$ is basepoint-free for general $p \in \operatorname{Pic}^{0}(Y)$, by Lemma \ref{lemma2}. Applying Proposition \ref{newpropo}, we conclude that $$|K_{X} + 2 h^{*}\Theta|$$ is basepoint-free.
\end{proof}
\bmhead{Acknowledgements}I am very grateful to  my advisors Steven Lu and Nathan Grieve for their constant help, invaluable advice and financial support. I would like to thank the anonymous referee and Sándor Kovács for their very careful reading of this work and for providing numerous comments and helpful suggestions.
\section*{Data availability}
Data sharing not applicable to this article as no data sets were generated or analyzed during the current study.
\section*{Declarations}
\textbf{Statement about conﬂict of interest} The author states that there is no conﬂict of interest.
\bibliography{Bibliography.bib}% common bib file
%% if required, the content of .bbl file can be included here once bbl is generated
%%\input sn-article.bbl
\end{document}